\documentclass[11pt,a4paper]{amsart}
\usepackage{amsmath,amssymb,amsthm,amsfonts}
\usepackage{mathrsfs}
\usepackage{amscd}
\usepackage{listings}
\usepackage{paralist}
\usepackage{booktabs}
\usepackage[usenames,dvipsnames]{color}
\usepackage{comment}
\usepackage{graphicx}
\usepackage{latexsym}
\usepackage[shortlabels]{enumitem}
\usepackage{thmtools}
\usepackage{url}
\usepackage[all,cmtip]{xy}
\usepackage[bookmarks=false]{hyperref} 

\newtheorem{theorem}{Theorem}[section]

\newtheorem{cor}[theorem]{Corollary}
\newenvironment{corollary}{\begin{cor} \em}{\end{cor}}
\newtheorem{tont}[theorem]{Definition}
\newenvironment{definition}{\begin{tont} \em}{\end{tont}}
\newtheorem{lemma}[theorem]{Lemma}

\newtheorem{proposition}[theorem]{Proposition}
 
\DeclareMathOperator\Ker{Ker}
\DeclareMathOperator\Ima{Im}
\begin{document}

\title [Associated primes and integral closure of rings]{Associated primes and integral closure of Noetherian rings}
\author{Antoni Rangachev}
\begin{abstract} 
Let $\mathcal{A}$ be a Noetherian ring and $\mathcal{B}$ be a finitely generated $\mathcal{A}$-algebra. Denote by $\overline{\mathcal{A}}$ the integral closure of $\mathcal{A}$ in $\mathcal{B}$. We give necessary and sufficient conditions for prime ideals to be in $\mathrm{Ass}_{\mathcal{A}}(\mathcal{B}/\overline{\mathcal{A}})$ and $\mathrm{Ass}_{\overline{\mathcal{A}}}(\mathcal{B}/\overline{\mathcal{A}})$ generalizing and strengthening classical results for rings of special type. 
\end{abstract}
\subjclass[2010]{13B22, 13A30, (13A18)}
\keywords{Integral closure of rings, associated primes, Rees algebras}
\address{Department of Mathematics\\
 University of Chicago \\
 Chicago, IL 60637\\
Institute of Mathematics \\
and Informatics, Bulgarian Academy of Sciences\\
Akad. G. Bonchev, Sofia 1113, Bulgaria\\}

\maketitle
\section{Introduction} 
Let $\mathcal{A} \subset \mathcal{B}$ be commutative rings with identity. Denote the integral closure of $\mathcal{A}$ in $\mathcal{B}$ by $\overline{\mathcal{A}}$. The main result of the paper is the following theorem. 

\begin{theorem}\label{main} Suppose $\mathcal{A}$ is Noetherian and $\mathcal{B}$ is a finitely generated $\mathcal{A}$-algebra. 
\vspace{.1cm}
\begin{enumerate}
\item[\rm{(i)}] Suppose that the minimal primes of $\mathcal{B}$ contract to primes of height at most one in $\mathcal{A}$. If $\mathfrak{q} \in \mathrm{Ass}_{\overline{\mathcal{A}}}(\mathcal{B}/\overline{\mathcal{A}})$, then $\mathrm{ht}(\mathfrak{q}) \leq 1$. If $\mathcal{B}$ is reduced and $\mathfrak{q}$ is not contained in a minimal prime of $\mathcal{B}$, then $\overline{\mathcal{A}}_{\mathfrak{q}}$ is a DVR. If $\mathfrak{p} \in \mathrm{Ass}_{\mathcal{A}}(\mathcal{B}/\overline{\mathcal{A}})$, then there exists a prime $\mathfrak{q}$ in $\overline{\mathcal{A}}$  with $\mathrm{ht}(\mathfrak{q}) \leq 1$ that contracts to $\mathfrak{p}$. 
\vspace{.2cm}
\item [\rm{(ii)}] Suppose that the minimal primes of $\mathcal{B}$ contract to minimal primes of $\mathcal{A}$ and  $\mathcal{A}_{\mathfrak{p}_{\mathrm{min}}}=\overline{\mathcal{A}}_{\mathfrak{p}_{\mathrm{min}}}$ for each minimal prime $\mathfrak{p}_{\mathrm{min}}$ of $\mathcal{A}$. Then there exists $f \in \mathcal{A}$ such that 
$$ \mathrm{Ass}_{\mathcal{A}}(\mathcal{B}/\overline{\mathcal{A}}) \subseteq \mathrm{Ass}_{\mathcal{A}}(\mathcal{B}_{f}/\mathcal{A}_f) \cup \mathrm{Ass}_{\mathcal{A}}(\mathcal{A}_{\mathrm{red}}/f\mathcal{A}_{\mathrm{red}}).$$ 
Furthermore, $\mathrm{Ass}_{\mathcal{A}}(\mathcal{B}/\overline{\mathcal{A}})$ and $\mathrm{Ass}_{\overline{\mathcal{A}}}(\mathcal{B}/\overline{\mathcal{A}})$ are finite.
\vspace{.1cm}
\item[\rm{(iii)}] Suppose $\mathcal{A}$ is locally equidimensional and universally catenary and suppose that the minimal primes of $\mathcal{B}$ contract to minimal primes of $\mathcal{A}$. If $\mathfrak{p} \in \mathrm{Ass}_{\mathcal{A}}(\mathcal{B}/\overline{\mathcal{A}})$, then $\mathrm{ht}(\mathfrak{p}) \leq 1$. 
\end{enumerate}
\end{theorem}

If $\mathcal{A}$ is reduced, and if $\mathcal{B}$ is contained in the ring of fractions of $\mathcal{A}$ and is integrally closed, then $\overline{\mathcal{A}}$ is the integral closure of $\mathcal{A}$ in its ring of fractions. In this case Thm.\ \ref{main} \rm{(i)} is part of the Mori--Nagata theorem (see Prp.\ 4.10.2 and Thm.\ 4.10.5 in \cite{Huneke}).

Suppose $R$ is a Noetherian ring and $I$ is an ideal in $R$, and $t$ is a variable. The {\it Rees algebra} of $I$, denoted $\mathcal{R}(I)$, is the subring of $R[t]$ defined as $\oplus_{n=0}^{\infty} I^{n}t^n$. Then Thm.\ \ref{main} \rm{(ii)} with $\mathcal{A}:= \mathcal{R}(I)$ and $\mathcal{B}:= R[t]$ along with Prp.\ \ref{finite} imply that the set $\mathrm{Ass}_{R}(R/\overline{I^n})$ is finite. This is a result of Rees \cite{Rees81}, which he derived as a consequence of his valuation theorem (cf.\ Sct.\ 8 in \cite{Sharp}), and of Ratliff \cite{Ratliff}, McAdam and Eakin \cite{Eakin}, who treated the case $\mathrm{ht}(I) \geq 1$.  In Cor.\ \ref{Rees finite} we show that Thm.\ \ref{main} \rm{(ii)} leads to an explicit description of $\mathrm{Ass}_{R}(R/\overline{I^n})$. An application of Thm.\ \ref{main} (\rm{ii}) to valuation theory will be considered in a sequal to this paper (\cite{Rangachev3}). 

Suppose $(R,\mathfrak{m})$ is a local Noetherian universally catenary ring. 
Theorem \ref{main} \rm{(iii)} was proved by various authors in the case when  $\mathcal{A}$ is the Rees algebra of an ideal or an $R$-module and $\mathcal{B}$ is a polynomial ring over $R$. 

Concretely, if $\mathcal{A}:= \mathcal{R}(I)$ is the Rees algebra of an ideal $I$ in  $R$ and $\mathcal{B}:=R[t]$ is the polynomial ring in one variable, then Thm.\ \ref{main} \rm{(iii)} is a classical result of McAdam \cite[Prp. 4.1]{McAdam}. His result was generalized by Katz and Rice \cite[Thm.\ 3.5.1]{Katz2} to the case when $\mathcal{A}$ is the Rees algebra of a finitely generated module $\mathcal{M}$ and $\mathcal{B}$ is the symmetric algebra of a free $R$-module $\mathcal{F}$ of finite rank that contains $\mathcal{M}$ with $\mathcal{M}$ and $\mathcal{F}$ generically equal. In  \cite[Thm.\ 5.4 ]{Rangachev} the author proved Thm.\ \ref{main}  \rm{(iii)} in the case when $\mathcal{A}$ and $\mathcal{B}$ are the Rees algebras of a pair of finitely generated $R$-modules $\mathcal{M} \subset \mathcal{N}$. In Cor.\ \ref{suv} we show that Thm.\ \ref{main}  \rm{(iii)} recovers at once a criterion for integral dependence of Simis, Ulrich and Vasconcelos ( \cite[Thm.\ 4.1]{SUV}). 
 




We prove a converse to Thm.\ \ref{main} (\rm{iii}) under some additional hypothesis on $\mathcal{B}$ without requiring $\mathcal{A}$ to be universally catenary. 

\begin{theorem}\label{converse} Let $(R,\mathfrak{m})$ be a local Noetherian ring contained in $\mathcal{A}$.  Suppose  $\mathfrak{m}\mathcal{B}$ is of height at least $2$. If  $\mathfrak{p}$ is a minimal prime of $\mathfrak{m}\mathcal{A}$ with $\mathrm{ht}(\mathfrak{p}) \leq 1$, then $\mathfrak{p} \in \mathrm{Ass}_{\mathcal{A}}(\mathcal{B}/\overline{\mathcal{A}})$. 
\end{theorem}
 
A typical situation when such $\mathcal{B}$ arises is when $\mathcal{A}$ is a Rees algebra of a module $\mathcal{M}$ that sits inside a free module $\mathcal{F}$ over a local ring $(R,\mathfrak{m})$ of dimension at least $2$. Then $\mathcal{B}$ can be taken to be the symmetric algebra of $\mathcal{F}$, which is a polynomial ring over $R$ and thus $\mathrm{ht}(\mathfrak{m}\mathcal{B}) \geq 2$.  
 
In this setup Thm.\ \ref{converse} along with  \cite[Prp.\ 8.5]{Rangachev}, which treats the case $\dim R=1$,  recovers results of  Burch \cite[Cor., pg. 373]{Burch} for Rees algebras of ideals, and of Katz and Rice \cite[Thm.\ 3.5.1]{Katz2}  for Rees algebras of modules embedded in free modules. More generally, in \cite[Thm.\ 8.3] {Rangachev} the author proved Thm. \ref{converse} assuming that $\mathcal{A}$ and $\mathcal{B}$ are standard graded $R$-algebras. 

In Prp.\ \ref{nice embedding} we prove that if $\mathcal{A}$ is a graded domain over $R$ with $\dim R \geq 2$, an embedding of $\mathcal{A}$ in a graded $R$-algebra $\mathcal{B}$ satisfying the hypothesis of Thm.\ \ref{converse} exists as a consequence of Noether normalization. Such embeddings along with Thm.\ \ref{main}  and Thm.\ \ref{converse} play an important role in the theory of local volumes of relatively ample invertible sheaves (see Scts.\ $2,3$ and $4$ in \cite{Rangachev2}).



{\bf  Acknowledgements.} I would like to thank Steven Kleiman and Madhav Nori for helpful and stimulating conversations, and for providing me with comments on improving the exposition of this paper, and the anonymous referee for suggesting a different proof of Thm.\ \ref{main}, which is outlined in the appendix to this paper, and for his useful comments. I was partially supported by the University of Chicago FACCTS grant ``Conormal and Arc Spaces in the Deformation Theory of Singularities.''

\section{Proofs}

We begin with two propositions. Prp.\ \ref{finite} generalizes Lem.\ 3.1 in \cite{Katz3} (cf.\ Prp.\ 7.1 in \cite{Rangachev}). It's central to the proof of Thm.\ \ref{main} \rm{(ii)} and Thm.\ \ref{converse}. The proof of Prp.\ \ref{finite} is based on the generic freeness lemma of Hochster and Roberts (see Lem.\ 8.1 in \cite{Hochster} or the lemma preceeding Thm.\ 24.1 in \cite{Matsumura}). 

The second proposition Prp.\ \ref{faithful flatness} and the corollary that follows it generalize Lem.\ 5.3 (1) in \cite{Rangachev}. We use these statements in the proof of Thm.\ \ref{main} to pass to the situation when $\mathcal{A}$ is a reduced Nagata (pseudo-geometric) ring, which guarantees that $\overline{\mathcal{A}}$ is Noetherian. 

Let $R$ be a ring, $\mathcal{M}$ a module. A prime ideal $\mathfrak{p}$ is said to be {\it associated} to $\mathcal{M}$ if there is a nonzero $m \in \mathcal{M}$ with $\mathfrak{p}=\mathrm{Ann}(m)$. The set of associated primes is denoted by $\mathrm{Ass}_{R}(\mathcal{M})$. 

\begin{proposition}\label{finite}
Let $R$ be a Noetherian ring and let $\mathcal{A} \subset \mathcal{B}$ be Noetherian $R$-algebras. Assume $\mathcal{B}$ is a finitely generated $\mathcal{A}$-algebra. Then  each prime in  $\mathrm{Ass}_{R}(\mathcal{B}/\mathcal{A})$ is a contraction of a prime in $\mathrm{Ass}_{\mathcal{A}}(\mathcal{B}/\mathcal{A})$. Furthermore, $\mathrm{Ass}_{\mathcal{A}}(\mathcal{B}/\mathcal{A})$ is finite. 
\end{proposition}
\begin{proof}  
The first part of the proposition is \cite[\href{http://stacks.math.columbia.edu/tag/05DZ}{Tag 05DZ}]{Stacks}. For completeness we include our own proof which is part of \cite[Prp.\ 7.1]{Rangachev}. Let $\mathfrak{q} \in \mathrm{Ass}_{R}(\mathcal{B}/\mathcal{A})$ and let $\tilde{b} \in \mathcal{B}/\mathcal{A}$ is such $\mathfrak{q} = \mathrm{Ann}_{R}(\tilde{b})$. Set $\mathcal{I}(\tilde{b})=\mathrm{Ann}_{\mathcal{A}}(\tilde{b})$. Then $\mathcal{I}(\tilde{b}) \cap R = \mathfrak{q}$. Because $\mathfrak{q}$ is prime in $R$, then the contraction of the radical of $\mathcal{I}(\tilde{b})$ to $R$ is $\mathfrak{q}$. But the radical of $\mathcal{I}(\tilde{b})$ is the intersection of finitely many primes in $\mathrm{Ass}_{\mathcal{A}}(\mathcal{B}/\mathcal{A})$. Thus, there exists a prime $Q \in \mathrm{Ass}_{\mathcal{A}}(\mathcal{B}/\mathcal{A})$  whose contraction to $R$ is $\mathfrak{q}$. 

Next, we show that $\mathrm{Ass}_{\mathcal{A}}(\mathcal{B}/\mathcal{A})$
is finite. Let $t_1, \ldots, t_k$ be the generators of $\mathcal{B}$ as an $\mathcal{A}$-algebra. Then by
breaking the filtration 
$$\mathcal{A} \subset \mathcal{A}[t_1] \subset \ldots 
\subset \mathcal{A}[t_1, \ldots, t_k] = \mathcal{B}$$
into short exact sequences, we see it's enough to show that $$\bigcup_{i=1}^{k}\mathrm{Ass}_{\mathcal{A}}(\mathcal{A}[t_1, \ldots, t_{i}]/\mathcal{A}[t_1, \ldots, t_{i-1}])$$ is finite. By the first part, it's enough to show that $\mathrm{Ass}_{\mathcal{A}[t_1,\ldots, t_{i-1}]}(\mathcal{A}[t_1, \ldots, t_{i}]/\mathcal{A}[t_1, \ldots, t_{i-1}])$ is finite. Hence it's enough to prove the finiteness part of the proposition for the case  $\mathcal{B}=\mathcal{A}[t]$ with $t \in \mathcal{B}$. Set $\mathcal{A}_i := \mathcal{A} + \mathcal{A}t+ \cdots + \mathcal{A}t^i$. Note that each $\mathcal{A}_i$ is a finitely generated $\mathcal{A}$-module. Consider the exact sequence 
$$\mathcal{A}_1/\mathcal{A}_0 \rightarrow \mathcal{A}_{2}/\mathcal{A}_{1} \rightarrow \cdots \rightarrow \mathcal{A}_{n+1}/\mathcal{A}_{n} \rightarrow \cdots$$
where each arrow is given by multiplication by $t$ and thus it's surjective. Denote by $\phi_{n}$ the composite map from $\mathcal{A}_1/\mathcal{A}_0$ to $\mathcal{A}_{n+1}/\mathcal{A}_{n}$. Then the chain
$$\mathrm{Ker}(\phi_1)\subset \ldots \subset \mathrm{Ker}(\phi_{n}) \subset \ldots$$
must stabilize eventually as $\mathcal{A}_1/\mathcal{A}_0$ is a Noetherian module. This shows that for $n \gg 0$
the map $\mathcal{A}_n/\mathcal{A}_{n-1} \rightarrow \mathcal{A}_{n+1}/\mathcal{A}_{n}$ is an isomorphism. Hence $\bigcup_{i=0}^{\infty} \mathrm{Ass}_{\mathcal{A}}(\mathcal{A}_{i+1}/\mathcal{A}_{i})$ is finite. 


Let $\mathfrak{p} \in \mathrm{Ass}_{\mathcal{A}}(\mathcal{B}/\mathcal{A})$. Let $b \in \mathcal{B}$ whose image $\tilde{b}$ in $\mathcal{B}/\mathcal{A}$ is such that $\mathrm{Ann}_{\mathcal{A}}(\tilde{b}) = \mathfrak{p}$. There exists $j \geq 1$ such that $b \in \mathcal{A}_j$. Then $\mathfrak{p} \in \mathrm{Ass}_{\mathcal{A}}(\mathcal{A}_j/\mathcal{A})$. But the latter set is contained $\bigcup_{i=0}^{\infty} \mathrm{Ass}_{\mathcal{A}}(\mathcal{A}_{i+1}/\mathcal{A}_{i})$. Thus $\mathrm{Ass}_{\mathcal{A}}(\mathcal{B}/\mathcal{A})$ is finite. 
\end{proof}

\begin{proposition}\label{faithful flatness}
Let $R$ be a ring and let $\mathcal{A} \subset \mathcal{B}$ be $R$-algebras. Let $R \rightarrow S$ be a faithfully flat ring map. Denote by $\overline{\mathcal{A} \otimes_{R}S}$ the integral closure of $\mathcal{A} \otimes_{R}S$ in $\mathcal{B} \otimes_{R}S$. Then 
$$\overline{\mathcal{A} \otimes_{R}S} \cap \mathcal{B} = \overline{\mathcal{A}}$$
\end{proposition}
\begin{proof}
By flatness of $R \rightarrow S$ it follows that $\mathcal{A} \otimes_{R} S$ is contained in $\mathcal{B} \otimes_{R} S$. By faithful flatness of $R \rightarrow S$ it follows that $\mathcal{A}$ and $\mathcal{B}$ inject into $\mathcal{A} \otimes_{R} S$ and $\mathcal{B} \otimes_{R} S$ respectively. Suppose $b \in \mathcal{B}$ is integral over $\mathcal{A} \otimes_{R}S$, i.e.\ $b$ satisfies the following relation 
\begin{equation}\label{int. dep.} 
b^n + a_{n-1}b^{n-1}+ \cdots + a_0 = 0
\end{equation}
where $a_i \in \mathcal{A} \otimes_{R}S$. Set $\mathcal{M}:= \mathcal{A}+ \mathcal{A}b+ \cdots + \mathcal{A}b^{n-1}$. Then $\mathcal{M} \subset \mathcal{A}[b]$. By (\ref{int. dep.}) we have $\mathcal{M} \otimes_{R}S=\mathcal{A}[b] \otimes_{R}S$, or 
$(\mathcal{A}[b]/\mathcal{M}) \otimes_{R}S=0$. Because $R \rightarrow S$ is faithfully flat, then $\mathcal{A}[b]=\mathcal{M}$ which implies that $b$ is integral over $\mathcal{A}$. Hence $\overline{\mathcal{A} \otimes_{R}S} \cap \mathcal{B} \subset \overline{\mathcal{A}}$. The opposite inclusion follows trivially from persistence of integral dependence. 
\end{proof}


\newpage
\vspace{1cm}
\begin{center}
{\it Proof of Theorem \ref{main}}
\end{center}
\vspace{.2cm}
Consider $\rm{(i)}$. We perform several reduction steps that will allow us to assume that $\mathcal{A}$ is a reduced local complete ring. Let $b \in \mathcal{B}$ be such that $\mathfrak{q}$ is the annihilator of the image of $b$ in $\mathcal{B}/\overline{\mathcal{A}}$. Set $\mathfrak{p}:=\mathfrak{q} \cap \mathcal{A}$. Suppose $\mathfrak{q}$ is contained in a minimal prime of $\mathcal{B}$. Then by hypothesis $\mathrm{ht}(\mathfrak{p}) \leq 1$. Thus by incomparability $\mathrm{ht}(\mathfrak{q}) \leq 1$ and we are done. So we can assume that $\mathfrak{q}$ is not contained in a minimal prime of $\mathcal{B}$. Then by prime avoidance we select $h \in \mathfrak{q}$ that avoids the minimal primes of $\mathcal{B}$. 
By Prop.\ 2.16 in \cite{Huneke} $\overline{\mathcal{A}_\mathfrak{p}}=\overline{\mathcal{A}}_{\mathfrak{p}}$. Consider $\mathcal{A}_{\mathfrak{p}} \subset \overline{\mathcal{A}_\mathfrak{p}} \subset \mathcal{B}_\mathfrak{p}$. Then $\mathfrak{q}\overline{\mathcal{A}_\mathfrak{p}}$ is associated to $\mathcal{B}_\mathfrak{p}/\overline{\mathcal{A}_\mathfrak{p}}$. Moreover, $\mathfrak{q}\overline{\mathcal{A}_\mathfrak{p}}$ is maximal because $\overline{\mathcal{A}_\mathfrak{p}}$ is integral over $\mathcal{A}_\mathfrak{p}$ and $\mathfrak{q}\overline{\mathcal{A}_\mathfrak{p}} \cap \mathcal{A}_\mathfrak{p} = \mathfrak{p}\mathcal{A}_\mathfrak{p}$. So we can assume that $\mathcal{A}$ is local with maximal ideal $\mathfrak{p}$.

Let $\widehat{\mathcal{A}}$ be the completion of $\mathcal{A}$ with respect to $\mathfrak{p}$. Set $\mathcal{A}':=\overline{\mathcal{A}} \otimes_{\mathcal{A}} \widehat{\mathcal{A}}$ and $\mathcal{B}':= \mathcal{B} \otimes_{\mathcal{A}} \widehat{\mathcal{A}}$. By flatness $\mathcal{A}'/\mathfrak{q}\mathcal{A}' = (\overline{\mathcal{A}}/\mathfrak{q}\overline{\mathcal{A}}) \otimes_{\mathcal{A}} \widehat{\mathcal{A}}.$ But $\overline{\mathcal{A}}/\mathfrak{q}\overline{\mathcal{A}} \otimes_{\mathcal{A}} \mathcal{A}/\mathfrak{p}=\overline{\mathcal{A}}/\mathfrak{q}\overline{\mathcal{A}}$ and $\mathcal{A}/\mathfrak{p} \otimes_{\mathcal{A}} \widehat{\mathcal{A}}=\mathcal{A}/\mathfrak{p}.$ Thus by the associativity of the tensor product $(\overline{\mathcal{A}}/\mathfrak{q}\overline{\mathcal{A}}) \otimes_{\mathcal{A}} \widehat{\mathcal{A}}= \overline{\mathcal{A}}/\mathfrak{q}\overline{\mathcal{A}}$. Therefore, $\mathfrak{q}\mathcal{A}'$ is maximal. 

Because $\mathcal{B} \rightarrow \mathcal{B}'$ and $\overline{\mathcal{A}} \rightarrow \mathcal{A}'$ are faithfully flat, then by Lem.\ B.1.3 and Prp.\ B.2.4 in \cite{Huneke} $h$ avoids the minimal primes of $\mathcal{B}'$ and $\mathrm{ht}(\mathfrak{q}\mathcal{A}') = \mathrm{ht}(\mathfrak{q})$. By Lem.\ 5.2 in \cite{Rangachev} we can replace $\widehat{\mathcal{A}}, \mathcal{A}'$ and $\mathcal{B}'$ by their reduced structures. The height hypothesis remain intact. Because $\widehat{\mathcal{A}}$ is a reduced complete local ring and $\mathcal{B}'$ is a finitely generated $\widehat{\mathcal{A}}$-algebra, then $\overline{\widehat{\mathcal{A}}}$ is module-finite over $\widehat{\mathcal{A}}$ by \cite[\href{http://stacks.math.columbia.edu/tag/03GH}{Tag 03GH}]{Stacks} (cf.\  Ex.\ 9.7 in \cite{Huneke} and \cite[\href{http://stacks.math.columbia.edu/tag/037J}{Tag 037J}]{Stacks}). But so is $\mathcal{A}'$ as $\mathcal{A}' \subset \overline{\widehat{\mathcal{A}}}$. In particular, $\mathcal{A}'$ is Noetherian and so $\mathfrak{q}\mathcal{A}'$ is finitely generated. 

Clearly $b  \cdot \mathfrak{q}\mathcal{A}' \subset \mathcal{A}'.$ So, either $b \in \mathcal{A}'$ or $\mathfrak{q}\mathcal{A}'$ is the annihilator of the 
the image of $b$ in $\mathcal{B}'/\mathcal{A}'$ viewed as an $\mathcal{A}'$-module. The former is impossible because if $b \in \mathcal{A}'$, then by Prp.\ \ref{faithful flatness} we would get $b \in \overline{\mathcal{A}}$.  

As $h$ avoids the minimal primes of $\mathcal{B}'$ and the latter is reduced, then $h$ is regular in $\mathcal{B}'$. 
Suppose $b\cdot \mathfrak{q}\mathcal{A}' \subset \mathfrak{q}\mathcal{A}'$. As $h \in \mathfrak{q}\mathcal{A}'$, then $\mathfrak{q}\mathcal{A}'$ is a finitely generated faithful  $\mathcal{A}'$-module. Then by the Determinantal Trick Lemma (see Lem.\ 2.1.8 in \cite{Huneke}) $b$ is integral over $\mathcal{A}'$. Thus by Prp.\ \ref{faithful flatness} we get $b \in \overline{\widehat{\mathcal{A}}} \cap \mathcal{B} = \overline{\mathcal{A}}$ which is impossible.

Therefore, $b\cdot \mathfrak{q}\mathcal{A}' \not \subset \mathfrak{q}\mathcal{A}'.$ Because $h$ is regular in $\mathcal{B}'$ we have
\begin{equation}\label{saturation}
\mathfrak{q}\mathcal{A}' = (h\mathcal{A}':_{ \mathcal{A}'}  hb).  
\end{equation}

Consider the localization $\mathcal{A}'_{\mathfrak{q}\mathcal{A}'}$ of $\mathcal{A}'$ at $\mathfrak{q}\mathcal{A}'$. Then there exists $z \in \mathfrak{q}$ such that $bz=\varepsilon$ where $\varepsilon$ is a unit in $\mathcal{A}'_{\mathfrak{q}\mathcal{A}'}$. This yields the identity $hbz=h\varepsilon$. Because $h$ is regular, then so is $hb$. Let $z' \in \mathfrak{q}\mathcal{A}'$. Then by (\ref{saturation}) there exists $u \in h\mathcal{A}'$ such that $hbz'=hu$ or $hb(z'-u\varepsilon^{-1}z)=0$. But $hb$ is regular, so $z'=u\varepsilon^{-1}z$. Thus $\mathfrak{q}\mathcal{A}'=(z)$, i.e.\ $\mathcal{A}'_{\mathfrak{q}\mathcal{A}'}$ is a DVR. 
In particular, $\mathrm{ht}(\mathfrak{q}\mathcal{A}')=\mathrm{ht}(\mathfrak{q})=1$. This completes the proof of the first part of $(\rm{i})$. 

Suppose $\mathcal{B}$ is reduced and $\mathfrak{q}$ is not contained in a minimal prime of $\mathcal{B}$. As above identify $\mathcal{A}$ with its localization at $\mathfrak{p}$. By what we just proved there exists $z \in \mathfrak{q}$ such that each generating set for the maximal ideal $\mathfrak{q}\mathcal{A}'$ consists of $z$ and nilpotents in $\mathcal{B}'$ up to multiplication with units. Consider the faithfully flat map 
$\overline{\mathcal{A}}_{\mathfrak{q}} \rightarrow \mathcal{A}'_{\mathfrak{q}\mathcal{A}'}$. An element in $\mathfrak{q}\overline{\mathcal{A}}_{\mathfrak{q}}$ is a nilpotent if and only if its image in $\mathcal{A}'_{\mathfrak{q}\mathcal{A}'}$ is nilpotent. But there are no nilpotents in $\mathfrak{q}\overline{\mathcal{A}}_{\mathfrak{q}}$ because $\mathcal{B}$ is reduced. Thus $\mathfrak{q}\overline{\mathcal{A}}_{\mathfrak{q}}=(z)$, i.e. $\overline{\mathcal{A}}_{\mathfrak{q}}$ is a DVR.

Next, suppose $\mathfrak{p} \in \mathrm{Ass}_{\mathcal{A}}(\mathcal{B}/\overline{\mathcal{A}})$. As above, we can assume that $\mathcal{A}$ is local with maximal ideal $\mathfrak{p}$. We want to show that there exists a maximal ideal $\mathfrak{q} \in \overline{\mathcal{A}}$ with $\mathrm{ht}(\mathfrak{q}) \leq 1$. 

Let $b_1 \in \mathcal{B}$ be such that the annihilator of the image of $b_1$ in $\mathcal{B}/\overline{\mathcal{A}}$ viewed as an $\mathcal{A}$-module is $\mathfrak{p}$. Preserve the notation from above. Let $\mathcal{I}$ be the annihilator of the image of $b_1$ in $\mathcal{B}'_{\mathrm{red}}/\mathcal{A}'_{\mathrm{red}}$ viewed as an $\mathcal{A}'_{\mathrm{red}}$-module. As $\mathcal{A}'_{\mathrm{red}}$ is Noetherian, there exists a primary decomposition of $\mathcal{I}$ in $\mathcal{A}'_{\mathrm{red}}$: 
$$\mathcal{I} = V_1 \cap \ldots \cap V_s.$$
Because $\mathcal{I}$ contains the maximal ideal  $\mathfrak{p}\widehat{\mathcal{A}}_{\mathrm{red}}$, each $V_i$ is primary to a maximal ideal $\mathfrak{m}_i$ in $\mathcal{A}'_{\mathrm{red}}$. By faithful flatness each $\mathfrak{m}_i$ is equal to $\mathfrak{q}_i\mathcal{A}'_{\mathrm{red}}$ where $\mathfrak{q}_i$ is a maximal ideal in $\overline{\mathcal{A}}$. Indeed, suppose $\mathfrak{m}_{i}'$ is maximal in $\mathcal{A}'$ such that $\mathfrak{m}_{i}'\mathcal{A}_{\mathrm{red}}' = \mathfrak{m}_i$. Then $\mathfrak{m}_{i}' \cap \mathcal{A} = \mathfrak{p}$ and so $\mathfrak{m}_{i}' \cap \overline{\mathcal{A}} = \mathfrak{q}_i$ where 
$\mathfrak{q}_i$ is maximal in $\overline{\mathcal{A}}$. But as shown above $\mathfrak{q}_i\mathcal{A}'$ is maximal and is contained in $\mathfrak{m}_{i}'$. Thus $\mathfrak{q}_i \mathcal{A}' = \mathfrak{m}_{i}'$ and so $\mathfrak{q}_i \mathcal{A}_{\mathrm{red}}' = \mathfrak{m}_i$.

For each $i=2, \ldots, s$ select $c_i \in \mathfrak{q}_{i}^{n_i}$ with $c_i \not \in \mathfrak{q}_1$ and $n_i$ sufficiently large so that $\Pi_{i=2}^{s} c_i$ is in $V_2 \cap \ldots \cap V_s$. Let $n_1$ be the smallest positive integer such $\mathfrak{m}_{1}^{n_1}=\mathfrak{q}_1^{n_1}\mathcal{A}'_{\mathrm{red}} \subset V_1$.
If $n_{1}>1$, select $c_1 \in \mathfrak{q}_{1}^{n_{1}-1}$ and $c_1 \not \in V_1$. If $n_1=1$, set $c_1:=1$. 
Set $c:= \Pi_{i=1}^{s}c_i$ and $b_2:=cb_1$. Then $b_2 \not \in \mathcal{A}'_{\mathrm{red}}$ as $c \not \in \mathcal{I}$. Thus the annihilator of $b_2$ in $\mathcal{B}'_{\mathrm{red}}/\mathcal{A}'_{\mathrm{red}}$ is $\mathfrak{m}_1$. Repeating the argument from above we get $\mathrm{ht}(\mathfrak{m}_1) \leq 1$. By faithful flatness $ \mathrm{ht}(\mathfrak{m}_1)= \mathrm{ht}(\mathfrak{q}_1)$. So $\mathrm{ht}(\mathfrak{q}_1)\leq 1$. This proves the existence of $\mathfrak{q}:= \mathfrak{q}_1$ with the desired properties. 

Consider $\rm{(ii)}$. By Prp.\ \ref{finite} there exist finitely many primes in $\mathrm{Ass}_{\mathcal{A}}(\mathcal{B}/\mathcal{A})$. But $\mathrm{Ass}_{\mathcal{A}}(\overline{\mathcal{A}}/\mathcal{A}) \subset \mathrm{Ass}_{\mathcal{A}}(\mathcal{B}/\mathcal{A})$. So $\mathrm{Ass}_{\mathcal{A}}(\overline{\mathcal{A}}/\mathcal{A})$ is finite, too. Because $\overline{\mathcal{A}}_{\mathfrak{p}_{\mathrm{min}}} = \mathcal{A}_{\mathfrak{p}_{\mathrm{min}}}$ for each minimal prime $\mathfrak{p}_{\mathrm{min}}$ of $\mathcal{A}$, then by prime avoidance we can select $f$ from the intersection of the minimal primes in $\mathrm{Ass}_{\mathcal{A}}(\overline{\mathcal{A}}/\mathcal{A})$ so that $f$ avoids each $\mathfrak{p}_{\mathrm{min}}$. Let $g \in \overline{\mathcal{A}}$. Then $\mathcal{A}[g]$ is module-finite over $\mathcal{A}$. Hence there exists $u$ such that $f^u$ annihilates 
$\mathcal{A}[g]/\mathcal{A}$. In particular, $f^{u}g \in \mathcal{A}$. Thus $\mathcal{A}_f=\overline{\mathcal{A}}_f$. 

Suppose $\mathfrak{p} \in \mathrm{Ass}_{\mathcal{A}}(\mathcal{B}/\overline{\mathcal{A}})$. If $f \not \in  \mathfrak{p}$, then $\mathfrak{p} \in \mathrm{Ass}_{\mathcal{A}}(\mathcal{B}_f/\overline{\mathcal{A}}_f)$. But $\mathcal{A}_f=\overline{\mathcal{A}}_f$. So $\mathfrak{p} \in \mathrm{Ass}_{\mathcal{A}}(\mathcal{B}_f/\mathcal{A}_f)$, which is a finite set by Prp.\ \ref{finite}.

Next suppose $f \in \mathfrak{p}$. Preserve the setup from the proof of part $\rm{(i)}$. Assume $\mathcal{A}$ is local with maximal ideal $\mathfrak{p}$. First we show that there exists a module-finite extension of $\mathcal{A}$ and a prime ideal $\mathfrak{P}$ in it of height at most one that contracts to $\mathfrak{p}$. By part $\rm{(i)}$ there exists $\mathfrak{q}_1 \in \overline{\mathcal{A}}$ with 
$\mathrm{ht}(\mathfrak{q}_1) \leq 1$ and $\mathfrak{q}_1 \cap \mathcal{A} = \mathfrak{p}$. By prime avoidance we can select $x \in \mathfrak{q}_1$ whose image in $\mathcal{A}'_{\mathrm{red}}$ avoids $\mathfrak{m}_i$ for $i=2, \ldots, s$. Consider $\mathcal{A}[x]$. Because $x$ is integral over $\mathcal{A}$, then $\mathcal{A}[x]$ is module-finite over $\mathcal{A}$. Set $\mathfrak{P}:= \mathfrak{q}_1 \cap \mathcal{A}[x]$. All maximal ideals in $\overline{\mathcal{A}}$ that contract to $\mathfrak{P}$ contain $x$. Thus they all have $\mathfrak{m}_1$ as their image in $\mathcal{A}'$. But $\mathrm{ht}(\mathfrak{m}_1) \leq 1$. So by faithful flatness all maximal ideals in $\overline{\mathcal{A}}$ that contract to $\mathfrak{P}$ are of height at most one. Consider the integral extension $\mathcal{A}[x]_{\mathfrak{P}} \hookrightarrow \overline{\mathcal{A}}_{\mathfrak{P}}$. Then 
$$\mathrm{ht}(\mathfrak{P})=\dim \mathcal{A}[x]_{\mathfrak{P}}=\dim \overline{\mathcal{A}}_{\mathfrak{P}} = \sup_{j \in J} \{\mathrm{ht}(\mathfrak{q}_j) \}$$
where $J$ is an index set for the maximal ideals in $\overline{\mathcal{A}}$ contracting to $\mathfrak{P}$. But $\mathrm{ht}(\mathfrak{q}_j) \leq 1$, so $\mathrm{ht}(\mathfrak{P}) \leq 1$. 

Because $f$ avoids the minimal primes of $\mathcal{A}$, then if there exists $d \in \mathcal{B}$ such that $fd=0$, then $d$ must be nilpotent. After passing to the reductions of $\mathcal{A}$ and $\mathcal{A}[x]$ the contraction of $\mathfrak{P}$ to $\mathcal{A}$ is still $\mathfrak{p}$ and the height of $\mathfrak{P}$ remains intact. To keep the notation simple we will identify $\mathcal{A}$ and $\mathcal{A}[x]$ with their corresponding reductions. Because $\mathcal{A}[x]$ is module-finite over $\mathcal{A}$ there exists a positive integer $n$ such that $f^{n}\mathcal{A}[x] \subset \mathcal{A}$. Observe that $f$ is regular in $\mathcal{A}[x]$ because $\mathcal{A}[x]$ is reduced. 

Next we proceed as in the proof of Lem.\ 4.9.5 in \cite{Huneke}. Set $\mathfrak{J}:= (\mathcal{A}: _{\mathcal{A}} \mathcal{A}[x])$. 
Suppose $\mathfrak{J} \not \subset \mathfrak{P}$. Then $\mathfrak{J}$ contains a unit, so $\mathcal{A}[x]=\mathcal{A}$ and $\mathfrak{P} =\mathfrak{p}$. But $\mathrm{ht}(\mathfrak{P}) \leq 1$. So $\mathrm{ht}(\mathfrak{p}) \leq 1$. But $f \in \mathfrak{p}$. Moreover, $f$ avoids the minimal primes of $\mathcal{B}$. So $f$ avoids the minimal primes of $\mathcal{A}$. Thus $\mathfrak{p}$ is an associated prime of $\mathcal{A}/f\mathcal{A}$. 

Assume $\mathfrak{J} \subset \mathfrak{P}$. Suppose that $\mathfrak{p}$ is not an associated prime of $\mathcal{A}/f^{n}\mathcal{A}$. Then there exists $z \in \mathfrak{p}$ that is regular in $\mathcal{A}/f^{n}\mathcal{A}$. Let $\mathfrak{P}=\mathfrak{P}_1, \ldots, \mathfrak{P}_k$ be the minimal primes  of $f^{n}\mathcal{A}[x]$. Select $e \in \mathfrak{P}_2 \cap \ldots \cap \mathfrak{P}_k$ and $e \not \in \mathfrak{P}$ such that $z^{l}e \in f^{n}\mathcal{A}[x]$ for some $l$. Then $$(f^{n}e\mathcal{A}[x])z^l = f^n(z^{l}e\mathcal{A}[x]) \subseteq f^{n}(f^{n}\mathcal{A}[x]) \subseteq f^{n}\mathcal{A}.$$
But $z$ is regular in $\mathcal{A}/f^{n}\mathcal{A}$. So $f^{n}e\mathcal{A}[x] \subseteq f^{n}\mathcal{A}$. But $f$ is regular in $\mathcal{A}[x]$, so $e\mathcal{A}[x] \subseteq \mathcal{A}$. Thus $e \in \mathfrak{J}$. But $e \not \in \mathfrak{P}$ which is a contradiction. Hence $\mathfrak{p} \in \mathrm{Ass}_{\mathcal{A}}(\mathcal{A}/f^{n}\mathcal{A})$. 
But $f$ is regular in $\mathcal{A}$. So $\mathrm{Ass}_{\mathcal{A}}(\mathcal{A}/f^{n}\mathcal{A})= \mathrm{Ass}_{\mathcal{A}}(\mathcal{A}/f\mathcal{A})$. Therefore,  $\mathfrak{p} \in \mathrm{Ass}_{\mathcal{A}}(\mathcal{A}/f\mathcal{A})$. 

Finally, we obtained that if $\mathfrak{p} \in \mathrm{Ass}_{\mathcal{A}}(\mathcal{B}/\overline{\mathcal{A}})$, then 
$$\mathrm{Ass}_{\mathcal{A}}(\mathcal{B}/\overline{\mathcal{A}}) \subseteq \mathrm{Ass}_{\mathcal{A}}(\mathcal{B}_{f}/\mathcal{A}_f) \cup \mathrm{Ass}_{\mathcal{A}}(\mathcal{A}_{\mathrm{red}}/f\mathcal{A}_{\mathrm{red}}).$$ 
Because $\mathrm{Ass}_{\mathcal{A}}(\mathcal{B}_{f}/\mathcal{A}_f)$ and $\mathrm{Ass}_{\mathcal{A}}(\mathcal{A}_{\mathrm{red}}/f\mathcal{A}_{\mathrm{red}})$ are finite, then so is $\mathrm{Ass}_{\mathcal{A}}(\mathcal{B}/\overline{\mathcal{A}})$. 

Next, we want to show that  $\mathrm{Ass}_{\overline{\mathcal{A}}}(\mathcal{B}/\overline{\mathcal{A}})$ is finite. Assuming again that $\mathcal{A}$ is local at $\mathfrak{p}$, we claim that the number of maximal ideals in  $\overline{\mathcal{A}}$ is finite. Indeed, there are only finitely many maximal ideals in $\mathcal{A}'$. Also, $\overline{\mathcal{A}} \rightarrow \mathcal{A}'$ is injective, and the ideal generated by the image of each maximal ideal in $\overline{\mathcal{A}}$ under this map is maximal in $\mathcal{A}'$ as shown in \rm{(i)}. Thus, for each prime $\mathfrak{p}$ in $\mathcal{A}$ there are only finitely many prime ideals in $\overline{\mathcal{A}}$ contracting to $\mathfrak{p}$. If $\mathfrak{q} \in \mathrm{Ass}_{\overline{\mathcal{A}}}(\mathcal{B}/\overline{\mathcal{A}})$, then $\mathfrak{q} \cap \mathcal{A}$ is in $\mathrm{Ass}_{\mathcal{A}}(\mathcal{B}/\overline{\mathcal{A}})$. But the latter set is finite. So $\mathrm{Ass}_{\overline{\mathcal{A}}}(\mathcal{B}/\overline{\mathcal{A}})$ is finite, too. 

Consider \rm{(iii)}. Let $\mathfrak{p} \in \mathrm{Ass}_{\mathcal{A}}(\mathcal{B}/\overline{\mathcal{A}})$. As before we can assume $\mathcal{A}$ is local with maximal ideal $\mathfrak{p}$. In the proof of part \rm{(ii)} we showed that there exists a module-finite extension $\mathcal{A}[x]$ and a maximal ideal $\mathfrak{P}$ of height at most one contracting to $\mathfrak{p}$. If $\mathfrak{P}$ is minimal, then $\mathfrak{p}$ is minimal by our hypothesis and we are done. Suppose $\mathrm{ht}(\mathfrak{P})=1$. Let $\mathfrak{q}_{\mathrm{min}}$ be a minimal prime of $\mathcal{A}[x]$ contained in $\mathfrak{P}$. Set $\mathfrak{p}_{\mathrm{min}}= \mathfrak{q}_{\mathrm{min}} \cap \mathcal{A}$. Because $\mathcal{A}$ is universally catenary, by the dimension formula (Thm.\ B.5.1 in \cite{Huneke}) applied to the extension $\mathcal{A}/\mathfrak{p}_{\mathrm{min}} \hookrightarrow \mathcal{A}[x]/\mathfrak{q}_{\mathrm{min}}$ we get
$$\mathrm{ht}(\mathfrak{P}(\mathcal{A}[x]/\mathfrak{q}_{\mathrm{min}}))=\mathrm{ht}(\mathfrak{p}(\mathcal{A}/\mathfrak{p}_{\mathrm{min}}))=1.$$
But $\mathcal{A}$ is equidimensional. So $\mathrm{ht}(\mathfrak{p})=1$. 
\qed
\vspace{.3cm}

The aim of the author was to give for the most part a self-contained proof of Thm.\ \ref{main}. The proof of part $\rm{(i)}$ mimics the proof of the well-known fact that in an integrally closed Noetherian ring the set of associated primes of an arbitrary principal ideal generated by a non-zero divisor consists exactly of the set of minimal primes of that ideal, and all such associated primes are locally principal (see Prp.\ 4.1.1 in \cite{Huneke}). The proof of part \rm{(ii)} uses some of the ideas of the proof of the Mori--Nagata theorem for Krull domains. The new ingredient here is Prp.\ \ref{finite} which proves the finiteness part of the statement. The proof of part \rm{(iii)} follows immediate from the proof of part  \rm{(ii)} and the dimension formula. 

One can give a concise proof of Thm.\ \ref{main} \rm{(i)} and the inclusion statement in part \rm{(ii)} using the theory of Krull domains and results of Ratliff as suggested by the referee. This is done in the appendix.

Let $R$ be a Noetherian ring and $I$ an ideal in $R$ and $t$ a variable. The {\it Rees algebra} of $I$, denoted $\mathcal{R}(I)$, is the subring of $R[t]$ defined as $\oplus_{n=0}^{\infty} I^{n}t^n$. Denote the $k$th graded pieces of $R[t]$ and $\mathcal{R}(I)$ by $R[t]_{k}$ and $\mathcal{R}(I)_k$, respectively. Denote by $\overline{I^n}$ the integral closure of $I^n$ in $R$. The integral closure 
$\overline{\mathcal{R}(I)}$ of $\mathcal{R}(I)$ in $R[t]$ is $\oplus_{n=0}^{\infty} \overline{I^{n}}t^n$. 

Note that if $P \in \mathrm{Ass}_{R}(R/\overline{I^n})$, then there exists a minimal prime $P_{\mathrm{min}}$ of $R$ with $P_{\mathrm{min}} \subseteq P$ such that $P/P_{\mathrm{min}}$ is associated to the integral closure of $I^{n}(R/P_{\mathrm{min}})$ in $R/P_{\mathrm{min}}$ (see Lem.\ 5.4.4 in \cite{Huneke}). So to keep the exposition as simple as possible we assume that $R$ is a domain.

\begin{corollary}\label{Rees finite} Let $R$ be a Noetherian domain, $I$ an ideal in $R$ and $a \in I$. Suppose $P \in \mathrm{Ass}_{R}(R/\overline{I^n})$
for some $n$. Then $P$ is a contraction of an associated prime of $a\mathcal{R}(I)$.
\end{corollary}

\begin{proof}
Apply Thm.\ \ref{main} \rm{(ii)} with $\mathcal{A}:= \mathcal{R}(I)$ and $\mathcal{B}:= R[t]$. 
Because $a^kR[t]_{k} \subset \mathcal{R}(I)_k$ for each $k \geq 1$ we can select $f:=a$. By  \cite[\href{http://stacks.math.columbia.edu/tag/05DZ}{Tag 05DZ}]{Stacks} or the proof of the second part of Prp.\ \ref{finite}, 
$P$ is a contraction of $\mathfrak{p} \in \mathrm{Ass}_{\mathcal{R}(I)}(R[t]/\overline{\mathcal{R}(I)})$. As $R[t]_{a}=\mathcal{R}(I)_a$ and $\mathcal{R}(I)$ is a domain, Thm.\ \ref{main} \rm{(ii)} implies that $\mathfrak{p}$ is associated to $a\mathcal{R}(I)$. 
\end{proof}



Next we show that Thm.\ \ref{main} \rm{(iii)} recovers a criterion for integral dependence due to Simis, Ulrich and Vasconcelos (Thm.\ 4.1 in \cite{SUV}). 
\begin{corollary}(Simis--Ulrich--Vasconcelos)\label{suv}
Let $\mathcal{A} \subset \mathcal{B}$ be an extension of rings with $\mathcal{A}$ Noetherian, locally equidimensional 
and universally catenary. Assume that $\mathcal{A}_{\mathfrak{p}} \subset \mathcal{B}_{\mathfrak{p}}$
is integral for each prime $\mathfrak{p}$ in $\mathcal{A}$ with $\mathrm{ht}(\mathfrak{p}) \leq 1$. Further, assume
that each minimal prime of $\mathcal{B}$ contracts to a minimal prime of $\mathcal{A}$. Then $\mathcal{B}$
is integral over $\mathcal{A}$. 
\end{corollary}
\begin{proof}
We can assume that $\mathcal{B}$ is reduced as each nilpotent element of $\mathcal{B}$ is integral over $\mathcal{A}$. Suppose $b \in \mathcal{B}$ and $b$ is not integral over $\mathcal{A}$. Denote by $\mathcal{A}[b]$ the algebra generated by $\mathcal{A}$ and $b$. Denote by $\overline{\mathcal{A}}$ the integral closure of $\mathcal{A}$ in $\mathcal{A}[b]$. Each minimal prime of $\mathcal{A}[b]$ is contracted from a minimal prime of $\mathcal{B}$ and thus each minimal prime of $\mathcal{A}[b]$ contracts to a minimal prime of $\mathcal{A}$. 
By Thm.\ \ref{main} \rm{(iii)} we know that the minimal primes in $\mathrm{Supp}_{\mathcal{A}}(\mathcal{A}[b]/\overline{\mathcal{A}})$ are
of height at most one. But by assumption none of the primes in $\mathcal{A}$ of height at most one is in $\mathrm{Supp}_{\mathcal{A}}(\mathcal{A}[b]/\overline{\mathcal{A}})$. 
We reached a contradiction. Thus $\mathcal{B}$ is integral over $\mathcal{A}$. 
\end{proof}

\begin{center}
{\it Proof of Theorem \ref{converse}}
\end{center}
\vspace{.1cm}
The proof is an improvement of the proof of Thm.\ 8.3 in \cite{Rangachev}. 

Suppose $\mathfrak{m} \not \in \mathrm{Ass}_{\mathcal{A}}(\mathcal{B}/\overline{\mathcal{A}})$. By prime avoidance and Prp.\ \ref{finite} we can select $h_1 \in \mathfrak{m}$ so that $h_1$ avoids all primes in $\mathrm{Ass}_{R}(\mathcal{B}/\mathcal{A}) \setminus \{\mathfrak{m}\}$, and the minimal primes of $\mathcal{A}$ and $\mathcal{B}$. Consider the map $$\psi_{h_1} \colon \mathcal{A}/h_{1}\mathcal{A} \rightarrow \mathcal{B}/h_{1}\mathcal{B}.$$ Set $\mathcal{A}(h_1):= \Ima \psi_{h_{1}}$. Observe that $\Ker \psi_{h_1} = (h_{1}\mathcal{B} \cap \mathcal{A})/h_{1}\mathcal{A}$. Suppose $b \in \mathcal{B}$ with $h_{1}b \in \mathcal{A}$. Then $b \in \overline{\mathcal{A}}$. Indeed, if $\mathcal{I}:= \mathrm{Ann}(\tilde{b})$ where $\tilde{b}$ is the image of $b$ in the $\mathcal{B}/\mathcal{A}$ viewed as an $R$-module, then the radical of $\mathcal{I}$ is the intersection of associated primes of $\mathcal{B}/\mathcal{A}$. But $h_1$ avoids all of them but $\mathfrak{m}$. Hence $\mathcal{I}$ must be $\mathfrak{m}$-primary. But $\mathfrak{m} \not \in \mathrm{Ass}_{\mathcal{A}}(\mathcal{B}/\overline{\mathcal{A}})$. So $b \in \overline{\mathcal{A}}$. Then
$$b^{s}+a_{1}b^{s-1}+\cdots+a_{s}=0$$
for some positive integer $s$ and $a_i \in \mathcal{A}$. Multiplying  both sides of the last equation by $h_{1}^s$ we obtain
$$(h_{1}b)^{s}+ha_{1}(h_{1}b)^{s-1}+\cdots+h_{1}^{s}a_{s}=0.$$
Thus $(h_{1}b)^s=0$ in $\mathcal{A}/h_{1}\mathcal{A}$. Then $\Ker \psi_{h_1}$ consists of nilpotents and thus $\mathcal{A}(h_1)$ and 
$\mathcal{A}/h_{1}\mathcal{A}$ have the same reduced structures. Because each minimal prime of $\mathfrak{m}\mathcal{B}$  is of height at least $2$, we can find $h_2 \in \mathfrak{m}$ such that $h_2$ avoids the minimal primes of $\mathcal{B}/h_{1}\mathcal{B}$. Because $\Ker \psi_{h_1}$ consists of nilpotents, $h_2$ avoids the minimal primes of $\mathcal{A}/h_{1}\mathcal{A}$. Thus each minimal prime of $\mathfrak{m}\mathcal{A}$ is of height at least $2$ which contradicts our assumption that there exists a minimal prime $\mathfrak{p}$ of $\mathfrak{m}\mathcal{A}$ of height at most one. Thus $\mathfrak{m} \in \mathrm{Ass}_{R}(\mathcal{B}/\overline{\mathcal{A}})$. Let $b$
such that $\mathrm{Ann}_{R}(\tilde{b}) = \mathfrak{m}$ where $\tilde{b}$ is the image of $b$ in $\mathcal{B}/\overline{\mathcal{A}}$. 
Then we can select $c \in \mathcal{A}$ so that $\mathrm{Ann}_{\mathcal{A}}(c \tilde{b})= \mathfrak{p}$ as we did in the proof of Thm.\ \ref{main} \rm{(i)}. 
\qed

Let $(R,\mathfrak{m})$ be a local Noetherian domain of dimension $d$. Let $\mathcal{M}$ be a nonzero finitely generated $R$-module contained in a free $R$-module $\mathcal{F}:=R^k$. Denote by $K$ the field of fractions of $R$. Set $e:= \mathrm{rk}(M\otimes_{R} K)$. The {\it Rees algebra} of $\mathcal{M}$, denoted by $\mathcal{R}(\mathcal{M})$, is defined as the symmetric algebra of $\mathcal{M}$ modulo its $R$-torsion submodule (for a more general definition see \cite{Eisenbud}). The symmetric algebra of $\mathcal{F}$ is the polynomial ring $R[y_1,\ldots, y_k]$.
So, $\mathcal{R}(\mathcal{M})$ can viewed as the subalgebra of $R[y_1, \ldots, y_k]$ generated by generators for $\mathcal{M}$ viewed as degree one polynomials in the $y_i$s. Denote by $\mathcal{F}^n$ and $\mathcal{M}^n$ the $n$th graded pieces of $\mathrm{Sym}(\mathcal{F})$ and $\mathcal{R}(\mathcal{M})$, respectively. Define the {\it integral closure of} $\mathcal{M}^n$ in $\mathcal{F}^n$, denoted by $\overline{\mathcal{M}^n}$, to be the module consisting of all  elements in $\mathcal{F}^n$ satisfying an integral dependence relation over $\mathcal{R}(\mathcal{M})$. 

The Krull dimension of $\mathcal{R}(\mathcal{M})$ is $d+e$. (see Lem.\ 16.2.2 (3) in \cite{Huneke}). The {\it special fiber} of $\mathcal{R}(\mathcal{M})$ is the ring $\mathcal{R}(\mathcal{M})/\mathfrak{m}\mathcal{R}(\mathcal{M})$. Its Krull dimension is called the {\it analytic spread} of $\mathcal{M}$ and is denoted by $l(\mathcal{M})$. If $\dim R >0$, then because $\mathcal{R}(\mathcal{M})$ is a domain $l(\mathcal{M}) \leq d+e-1$. The following result is a direct corollary of Thm.\ \ref{main} \rm{(iii)} and Thm.\ \ref{converse}. It gives a necessary and sufficient condition for $\mathfrak{m}$ to be in $\mathrm{Ass}_{R}(\mathcal{F}^n/\overline{\mathcal{M}^n})$. Note that when $R$ is a field $\mathcal{M}$ and $\mathcal{F}$ are $R$-vector spaces. Thus $\mathcal{R}(\mathcal{M})$ is a polynomial ring over $R$; hence $\mathcal{M}^n=\overline{\mathcal{M}^n}$ for each $n$. Therefore, $\mathfrak{m}  \in \mathrm{Ass}_{R}(\mathcal{F}^n/\overline{\mathcal{M}^n})$ if and only if $e \neq k$. 

\begin{corollary}\label{classical} Suppose $(R,\mathfrak{m})$ is a local Noetherian domain of positive dimension. If $R$ is universally catenary and $\mathfrak{m} \in \mathrm{Ass}_{R}(\mathcal{F}^n/\overline{\mathcal{M}^n})$ for some $n$, then $l(M)=d+e-1$. Conversely, if $l(M)=d+e-1$, then $\mathfrak{m} \in \mathrm{Ass}_{R}(\mathcal{F}^n/\overline{\mathcal{M}^n})$ for some $n$, unless $\dim R=1$ and $\mathcal{M}$ is free and a direct summand of $\mathcal{F}$. 
\end{corollary} 
\begin{proof} Set $\mathcal{A}:= \mathcal{R}(\mathcal{M})$ and $\mathcal{B}:=\mathrm{Sym}(\mathcal{F})$. Suppose $R$ is universally catenary and $\mathfrak{m} \in \mathrm{Ass}_{R}(\mathcal{F}^n/\overline{\mathcal{M}^n})$ for some $n$. Because $R$ is universally catenary so is $\mathcal{A}$. Because $\mathfrak{m} \in \mathrm{Ass}_{R}(\mathcal{B}/\overline{\mathcal{A}})$, then by \cite[\href{http://stacks.math.columbia.edu/tag/05DZ}{Tag 05DZ}]{Stacks} or the proof of Prp.\ \ref{finite} there exists $\mathfrak{p} \in \mathcal{A}$ such that $\mathfrak{p} \cap R=\mathfrak{m}$ and $\mathfrak{p} \in \mathrm{Ass}_{\mathcal{A}}(\mathcal{B}/\overline{\mathcal{A}})$. Then by Thm.\ \ref{main} \rm{(iii)} $\mathrm{ht}(\mathfrak{p})\leq 1$. 
Applying the dimension formula (Thm.\ B.5.1 in \cite{Huneke}) for $R$ and $\mathcal{A}$ we obtain $\dim \mathcal{A}/\mathfrak{p}=d+e-1$. But $l(M) \geq \dim \mathcal{A}/\mathfrak{p}$ because $\mathfrak{p}$ contains $\mathfrak{m}$. Hence $l(M) = d+e-1$. 

Conversely, assume $l(\mathcal{M})=d+e-1$. Suppose $\dim R=1$. If $\mathfrak{m} \not \in \mathrm{Ass}_{R}(\mathcal{F}^n/\overline{\mathcal{M}^n})$ for each $n$, then by Prp.\ 8.5 in \cite{Rangachev} we conclude that $\mathcal{M}$ is free of rank $e$ and a direct summand of $\mathcal{F}$. Assume $\dim R > 1$. Then $\mathrm{ht}(\mathfrak{m}\mathcal{B}) \geq 2$ because $\mathcal{B}$ is a polynomial ring over $R$. Because $l(\mathcal{M})=d+e-1$, by the the dimension inequality (see Thm.\ B.2.5 in \cite{Huneke}) there exists a minimal prime $\mathfrak{p}$ of $\mathfrak{m}\mathcal{R}(\mathcal{M})$ with $\mathrm{ht}(\mathfrak{p}) \leq 1$. Thus by Thm.\ \ref{converse} $\mathfrak{p} \in \mathrm{Ass}_{\mathcal{A}}(\mathcal{B}/\overline{\mathcal{A}})$ which yields by \cite[\href{http://stacks.math.columbia.edu/tag/05DZ}{Tag 05DZ}]{Stacks} $\mathfrak{p} \cap R = \mathfrak{m}$ is in $\mathrm{Ass}_{R}(\mathcal{B}/\overline{\mathcal{A}})$. By degree considerations we get $\mathfrak{m} \in \mathrm{Ass}_{R}(\mathcal{F}^n/\overline{\mathcal{M}^n})$ for some $n$ (in fact for all sufficiently large $n$ by Thm.\ 7.5 in \cite{Rangachev}).
\end{proof}
When $\mathcal{M}$ is an ideal in $R$, Cor.\ \ref{classical} is a result of McAdam \cite[Prp.\ 4.1]{McAdam} and 
Burch \cite[Cor., pg.\ 373]{Burch} (cf.\ Thm.\ 5.4.6 and Thm.\ 5.4.7 in \cite{Huneke}). 
In the more general module setup Cor.\ \ref{classical}
is due to Katz and Rice \cite[Thm.\ 3.5.1]{Katz2} for the case $e=k$. Cor.\ \ref{classical} was proved by the author 
(see \cite[Cor.\ 8.9 ]{Rangachev}) assuming that $l(\mathcal{F})<d+e-1$ whenever $\dim R \geq 2$, which is weaker than the assumption that $\mathcal{F}$ is free. 

\vspace{.1cm}
When $\mathcal{A}$ is the Rees algebra of a module, then $\mathcal{A}$ comes equipped with an embedding in a polynomial ring $\mathcal{B}$. For a general finitely generated $R$-algebra $\mathcal{A}$ an embedding into a polynomial ring over $R$ may not exist. However, in the next proposition we show that under mild assumptions, $\mathcal{A}$ always has an embedding in a finitely generated graded $R$-algebra $\mathcal{B}$ satisfying the hypothesis of Thm.\ \ref{converse} provided that $\dim R \geq 2$. The existence of such embedding plays a crucial role in the excesss-degree formula in Thm.\ 2.4 in \cite{Rangachev2}, the computation of the local volume of a relatively ample invertible sheaf in \cite[Sct.3]{Rangachev2} and in the proof of the Local Volume Formula in \cite[Sct.4]{Rangachev2}.

\begin{proposition}\label{nice embedding}
 Suppose $R$ is a reduced equidimensional universally catenary Noetherian ring of positive dimension or an infinite field. Assume $\mathcal{A} = \oplus_{i=0}^{\infty} \mathcal{A}_{i}$ is a reduced equidimensional standard graded algebra over $R$. Assume that the minimal primes of $\mathcal{A}$ contract to minimal primes of $R$. Then there exists a standard graded $R$-algebra $\mathcal{B} = \oplus_{i=0}^{\infty} \mathcal{B}_{i}$ such that 
\begin{enumerate}
    \item[\rm{(1)}]$\mathcal{B}$ is a birational extension of $\mathcal{A}$, and the inclusion $\mathcal{A} 
 \subset \mathcal{B}$ is  homogeneous;
   
    \item[\rm{(2)}] For each prime $\mathfrak{p}$ in $R$ the minimal primes of $\mathfrak{p}\mathcal{B}$ are of height at least $\mathrm{ht}(\mathfrak{p}/\mathfrak{p}_{\mathrm{min}})$ where $\mathfrak{p}_{\mathrm{min}}$ is a minimal prime of $R$ contained in $\mathfrak{p}$.
\end{enumerate}
\end{proposition}
\begin{proof} Denote by $K$ the total ring of fractions of $R$.
Because $\mathcal{A}$ is reduced and its minimal primes contract to minimal primes of $R$, then $\mathcal{A}$ is $R$-torsion free. Hence $\mathcal{A}$ injects into $\mathcal{A} \otimes K$. Set $e:= \dim \mathcal{A} \otimes K$. Let $\mathfrak{p}_1, \ldots, \mathfrak{p}_l$ be the minimal primes of $R$ and $\mathfrak{q}_1, \ldots, \mathfrak{q}_t$ be the minimal primes of $\mathcal{A}$. Fix a minimal prime $\mathfrak{q}_u$. Assume $\mathfrak{q}_u$ contracts to  $\mathfrak{p}_i$. Set $\kappa(\mathfrak{q}_u):=\mathrm{Frac}(\mathcal{A}/\mathfrak{q}_u)$ and $\kappa(\mathfrak{p}_i):=\mathrm{Frac}(R/\mathfrak{p}_i)$. Then by \cite[\href{http://stacks.math.columbia.edu/tag/02JX}{Tag 02JX}]{Stacks} or Lem.\ 3.1 (ii) in \cite{KT-Al} we get
$$\dim \mathcal{A}/\mathfrak{q}_u=\dim R/\mathfrak{p}_i + \mathrm{tr.\ deg}_{\kappa(\mathfrak{p}_i)}\kappa(\mathfrak{q}_u).$$
Because $\mathcal{A}$ and $R$ are equidimensional we obtain that $\mathcal{A}_{\mathfrak{p}_i}$ is equidimensional with $\dim A_{\mathfrak{p}_i}= \dim \mathcal{A}/\mathfrak{q}_{u}-\dim R/\mathfrak{p}_i=e$. Because $R$ is reduced, $K=R_{\mathfrak{p}_1} \times \cdots \times R_{\mathfrak{p}_l}$. Thus $\mathcal{A} \otimes K = \mathcal{A}_{\mathfrak{p}_1} \times \cdots \times \mathcal{A}_{\mathfrak{p}_l}$. 

For each $i=1, \ldots, l$ denote by $\pi_{i}$ the projection homomorphism $\mathcal{A} \otimes K \rightarrow \mathcal{A}_{\mathfrak{p}_i}$. Because $R$ is equdimensional of positive dimension or $R$ is an infinite field, then each field $R_{\mathfrak{p}_i}$ is infinite. Thus by Noether normalization we can select $e$ elements $b_{1}', \ldots, b_{e}'$ in $\mathcal{A}_1 \otimes K$ such that $\mathcal{A}_{\mathfrak{p}_i}$ is integral over $R_{\mathfrak{p}_i}[\pi_{i}(b_{1}'), \ldots, \pi_{i}(b_{e}')]$ for each $i$. 

Let $a_1, \ldots, a_s$ be degree one generators of  $\mathcal{A}$ over $R$. Then each $\pi_{i}(a_j)$ for $j=1, \ldots, s$ satisfies
an equation of integral dependence over $R_{\mathfrak{p}_i}[\pi_{i}(b_{1}'), \ldots, \pi_{i}(b_{e}')]$. For each $i=1, \ldots, l$ let $d_i \in R_{\mathfrak{p}_i}$ be the product over $j$ of all (nonzero) denominators appearing in the relation of integral dependence
of $\pi_{i}(a_j)$. For each $k=1, \ldots, e$ set
$$b_{k}:= \Big( \frac{\pi_{1}(b_{k}')}{d_1}, \ldots, \frac{\pi_{l}(b_{k}')}{d_l} \Big) \ \text{and} \ \mathcal{B}:= \mathcal{A}[b_1, \ldots, b_e].$$

As each $b_k$ is a fraction with enumerator in $\mathcal{A}_1$ and denominator in $R$, then $\mathcal{B}$ inherits naturally a grading from $\mathcal{A}$ with $\deg (b_k)=1$ for each $k=1, \ldots e$. Thus $\mathcal{A} \subset \mathcal{B}$ is a homogeneous inclusion. Because $b_{k}$ is in the total ring of fractions of $\mathcal{A}$, then $\mathcal{B}$ is a birational extension of $\mathcal{A}$. This proves $\rm{(1)}$.

Observe that $\mathcal{B}$ is integral over $R[b_1, \ldots, b_e]$ because each $a_j$ for $j=1, \ldots, s$ is integral over $R[b_1, \ldots, b_e]$. Note that $\mathfrak{p}\mathcal{B} \neq \mathcal{B}$. Indeed, 
$\mathcal{A}_0 = R$, so
$\mathcal{B}_0 = R$ and so $\mathfrak{p}\mathcal{B}_0 = \mathfrak{p} \neq R$. Let $Q$ be a minimal prime of $\mathfrak{p}\mathcal{B}$. Set $\mathfrak{p}':=Q \cap R$. Denote by $\widetilde{b_k}$ the images of $b_k$ in $\mathcal{B}/Q\mathcal{B}$. Then
\begin{equation}\label{int. dim.}
(R/\mathfrak{p}')[\widetilde{b_1}, \ldots, \widetilde{b_e}] \hookrightarrow \mathcal{B}/Q\mathcal{B}
\end{equation}
is an integral extension. Thus, $\mathrm{tr.\ deg}_{\kappa(\mathfrak{p}')}\kappa(Q) \leq e$, where $\kappa(\mathfrak{p}'):= \mathrm{Frac}(R/\mathfrak{p}')$ and $\kappa(Q):= \mathrm{Frac}(\mathcal{B}/Q\mathcal{B})$.  

Let $Q_{\mathrm{min}}$ be a minimal prime of $\mathcal{B}$ contained in $Q$. Set $\mathfrak{p}_{\mathrm{min}}:= Q_{\mathrm{min}} \cap R$. Because $R$ is a universally catenary and $\mathcal{B}$ is a finitely generated over $R$, the dimension formula (Thm.\ B.5.1 in \cite{Huneke}) gives  
$$\mathrm{ht}(Q/Q_{\mathrm{min}}) + \mathrm{tr.\ deg}_{\kappa(\mathfrak{p}')}\kappa(Q) = \mathrm{ht}(\mathfrak{p}'/\mathfrak{p}_{\mathrm{min}}) + \mathrm{tr.\ deg}_{\kappa(R/\mathfrak{p}_{\mathrm{min}})}\kappa(\mathcal{B}/Q_{\mathrm{min}})= \mathrm{ht}(\mathfrak{p}'/\mathfrak{p}_{\mathrm{min}}) + e.$$
But $\mathrm{tr.\ deg}_{\kappa(\mathfrak{p}')}\kappa(Q) \leq e$ and $\mathfrak{p} \subset \mathfrak{p}'$. Therefore,  $\mathrm{ht}(Q/Q_{\mathrm{min}}) \geq \mathrm{ht}(\mathfrak{p}'/\mathfrak{p}_{\mathrm{min}}) \geq \mathrm{ht}(\mathfrak{p}/\mathfrak{p}_{\mathrm{min}})$. 
\end{proof}



\section{Appendix}
Here we present a proof of part $\rm{(i)}$ and the inclusion statement in part $\rm{(ii)}$ of Thm.\ \ref{main} kindly suggested to the authour by the referee which uses the theory of Krull domains (see Sct.\ 4.10 in \cite{Huneke}) and results of Ratliff \cite{Ratliff2}. 

First we introduce the notion of Krull domains (see Dfn.\ 4.10.1 in \cite{Huneke}). 
\begin{definition} An integral domain $R$ is a {\it Krull domain}
if 
\begin{enumerate}
    \item [(i)] for every prime ideal $P$ of $R$ with $\mathrm{ht}(P)=1,$ the localization $R_{P}$ is a Noetherian integrally closed domain, 
    \item [(ii)] $R = \cap_{\mathrm{ht}(P)=1} R_{P}$, and

    \item[(iii)] each nonzero $x \in R$ lies in finitely many prime ideals of $R$ of height one. 
\end{enumerate}
\end{definition}

The Mori--Nagata theorem \cite[Thm.\ 4.10.5]{Huneke} says that the integral closure of a reduced Noetherian ring $S$ in its total ring of fractions is a direct product of $l$ Krull domains where $l$ is the number of minimal primes of $S$. 

In \cite{Ratliff2} Ratliff extends the definition of Krull domains to the nondomain case as follows. 

\begin{definition}\label{rat} Define the class $\mathfrak{C}$ of commutative rings $R$ satisfying the following properties:
\begin{itemize}
    \item [(i)] $R$ has finitely many minimal prime, and
    \item [(ii)] for each minimal prime $P_{\mathrm{min}}$ the integral closure of $R/P_{\mathrm{min}}$ is a Krull domain. 
\end{itemize}
\end{definition}

First, we show that $\overline{\mathcal{A}}$ belongs to the class $\mathfrak{C}$. Because $\mathcal{B}$ is Noetherian and because each minimal prime of $\mathcal{\overline{A}}$ is a contraction of a minimal prime of $\mathcal{B}$, then $\mathcal{\overline{A}}$ has finitely many primes and thus satisfies \rm{(i)} in Dfn.\ \ref{rat}. Suppose $\mathfrak{q}_{\mathrm{min}}$ is a minimal prime of $\mathcal{\overline{A}}$. We have $\mathfrak{q}_{\mathrm{min}}=\mathfrak{q}_{0} \cap \mathcal{B}$ where $\mathfrak{q}_{0}$ is a minimal prime of $\mathcal{B}$. Consider the nested chain of domains $\mathcal{A}/(\mathfrak{q}_{0} \cap \mathcal{A}) \subset \overline{\mathcal{A}}/\mathfrak{q}_{\mathrm{min}} \subset \mathcal{B}/\mathfrak{q}_{0}$. Note that $\overline{\mathcal{A}}/\mathfrak{q}_{\mathrm{min}}$ is integral over $\mathcal{A}/(\mathfrak{q}_{0} \cap \mathcal{A})$ and therefore the two domains have the same integral closure in $\mathcal{B}/\mathfrak{q}_{0}$. Let $E$ be the algebraic closure of 
$\mathrm{Frac}(\mathcal{A}/(\mathfrak{q}_{0} \cap \mathcal{A}))$ in $\mathrm{Frac}(\mathcal{B}/\mathfrak{q}_{0})$. Note that $E$ is a finite field extension of $\mathrm{Frac}(\mathcal{A}/(\mathfrak{q}_{0} \cap \mathcal{A}))$. Denote by $L$ the integral closure of $\mathcal{A}/(\mathfrak{q}_{0} \cap \mathcal{A})$ in $E$. By the Mori--Nagata theorem $L$ is a Krull domain (see Ex.\ 4.15 in \cite{Huneke}). But $L$ is also the integral closure of $\overline{\mathcal{A}}/\mathfrak{q}_{\mathrm{min}}$ in its field of fractions. Thus $\overline{\mathcal{A}}$ belongs to the class $\mathfrak{C}$.


Consider Thm.\ \ref{main} \rm{(i)}. Note that by killing the nilradical of $\mathcal{B}$ we do not alter anything in the statement of  Thm.\ \ref{main} \rm{(i)}. Let's assume for now that $\mathcal{B}$ and hence $\mathcal{A}$ and $\overline{\mathcal{A}}$ are reduced. Suppose $\mathfrak{q} \in \mathrm{Ass}_{\overline{\mathcal{A}}}(\mathcal{B}/\overline{\mathcal{A}})$. Write $\mathfrak{q} = (\overline{\mathcal{A}} \colon b)$ for some $b \in \mathcal{B}$. If $\mathfrak{q}$ is contained in a minimal prime of $\mathcal{B}$, then by hypothesis and incomparability $\mathrm{ht}(\mathfrak{q}) \leq 1$ and we are done. Thus we can assume that $\mathfrak{q}$ avoids the minimal primes of $\mathcal{B}$. Further, we may assume $\overline{\mathcal{A}}$ is local at $\mathfrak{q}$. As in (\ref{saturation}) select a regular element $h \in \mathfrak{q}$. Then 
$$\mathfrak{q}=(h\overline{A} \colon hb).$$
But $\overline{\mathcal{A}}$ belongs to the class $\mathfrak{C}$, so by Lem.\ 2.6 (2) and Prp.\ 2.7 in \cite{Ratliff2} the reduced 
local ring $\overline{\mathcal{A}}$ is a DVR.

In general if $\mathcal{B}$ is not reduced, then its nilpotents belong to $\overline{\mathcal{A}}$. So locally at $\mathfrak{q}$ the ring $\overline{\mathcal{A}}$ may fail to be Noetherian and thus one can not hope for it to be a DVR, however, it is very close to one. 

\begin{definition}\cite[pg.213]{Ratliff2} 
A {\it quasi-v-ring} $R$ is a local possibly non-Noetherian ring of dimension one such that its maximal ideal is generated by a regular element in $R$. 
\end{definition}
Thus, by the aforementioned results of Ratliff, in the nonreduced case the conclusion is that locally at $\mathfrak{q}$ the ring $\overline{\mathcal{A}}$ is a quasi-v-ring. 

For the second statement in Thm.\ \ref{main}, one proceeds similarly. Suppose $\mathfrak{p} \in \mathrm{Ass}_{\mathcal{A}}(\mathcal{B}/\overline{\mathcal{A}})$. If $\mathfrak{p}$ is in a minimal prime of $\mathcal{B}$, then $\mathrm{ht}(\mathfrak{p}) \leq 1$ and we are done. So we can assume $\mathfrak{p}$ avoids the minimal primes of $\mathcal{B}$. First, suppose $\mathcal{B}$ is reduced. We may assume $\mathcal{A}$ is local at $\mathfrak{p}$. Because $\mathfrak{p}=(\overline{\mathcal{A}} \colon b)$ for some $b \in \mathcal{B}$, then  $\mathfrak{p}=(h\overline{A} \colon hb)$ for $h \in \mathfrak{p}$ a nonzero divisor in $\mathcal{B}$. Thus $\mathfrak{p}$ is the contraction of $(h\overline{\mathcal{A}} \colon_{\overline{\mathcal{A}}} hb) = (\overline{h\overline{\mathcal{A}}} \colon_{\overline{\mathcal{A}}} hb)$, so by \cite[Thm.\ 2.10]{Ratliff2}
$\mathfrak{p}$ is the contraction of a prime $\mathfrak{q} \subset \overline{\mathcal{A}}$ such that $\overline{\mathcal{A}}_{\mathfrak{q}}$ is a DVR. In general, if $\mathcal{B}$ is nonreduced, then as before $\mathfrak{p}$ is a contraction of $\mathfrak{q} \subset \overline{\mathcal{A}}$ such that  $\overline{\mathcal{A}}_{\mathfrak{q}}$ is a qusi-v-ring. 

Next, consider the statement in part $\rm{(ii)}$ of Thm.\ \ref{main} about inclusions of sets of associated primes. Again assume that $\mathcal{B}$ is reduced. Because $\mathrm{Ass}_{\mathcal{A}}(\overline{\mathcal{A}}/\mathcal{A})$ is finite by Prp.\ \ref{finite} and because $\mathcal{A}$ and $\overline{\mathcal{A}}$ are equal locally at the minimal primes of $\mathcal{A}$ we can select $f \in \mathcal{A}$ such that $f$ is in all primes in $\mathrm{Ass}_{\mathcal{A}}(\overline{\mathcal{A}}/\mathcal{A})$ and $f$ avoids the minimal primes of $\mathcal{A}$. After killing the nilradical as before, any $\mathfrak{p} \in \mathrm{Ass}_{\mathcal{A}}(\mathcal{B}/\overline{\mathcal{A}})$, which is not minimal, is the contraction of a prime $\mathfrak{q} \subset \overline{\mathcal{A}}$ such that $\overline{\mathcal{A}}_{\mathfrak{q}}$ is a DVR. Because $\mathcal{A}$ and $\overline{\mathcal{A}}$ have the same ring of fractions, then one can see that $\mathfrak{p}$ is a contraction of prime divisor of $f\tilde{\mathcal{A}}$ where $\tilde{\mathcal{A}}$ is the integral closure of $\mathcal{A}$ in its ring of fractions. There are only finitely many primes in $\tilde{\mathcal{A}}$ lying over $\mathfrak{p}$ by \cite[Prp.\ 4.8.2]{Huneke}. Then by \cite[Thm.\ 2.15]{Ratliff2} we conclude that $\mathfrak{p}$ is a prime divisor of $f\mathcal{A}$. If $\mathfrak{p}$ is a minimal prime, then $\mathfrak{p} \in \mathrm{Ass}_{A}(\mathcal{B}_f/\mathcal{A}_f)$ because $f \not \in \mathfrak{p}$ and $\mathfrak{p} \in \mathrm{Ass}_{\mathcal{A}}(\mathcal{B}/\mathcal{A})$ by hypothesis. 



\begin{thebibliography}{abcdef}
\bibitem[Bur72]{Burch}
Burch, L., {\it Codimension and analytic spread,} Proc. Cambridge Phil. Soc., {\bf 72} (1972), 369--373.

\bibitem[EHU03]{Eisenbud}
Eisenbud, D., Huneke, C., and Ulrich, B., {\it What is the Rees algebra of a module}, Proc. Amer. Math. Soc., {\bf 131} (2003), no. 3, 701--708.

\bibitem[HR74]{Hochster}
Hochster, M., Roberts, J., {\it Rings of invariants of reductive groups acting on regular rings are Cohen-Macaulay,} Adv.\ in Math.\
{\bf 13} (1974), 115--175. 


\bibitem[KaR08]{Katz2}
Katz, D, Rice, G., {\it Asymptotic prime divisors of torsion-free symmetric powers of a module,} Journal of Algebra, {\bf 319} (2008), 2209-–2234.

\bibitem[KaP13]{Katz3}
Katz, D., Puthenpurakal, T., {\it Quasi-finite modules and asymptotic prime
divisors,} Journal of Algebra {\bf 380} (2013), 18--29.


\bibitem[KT94]{KT-Al}
Kleiman, S., Thorup, A., {\it  A geometric theory of the Buchsbaum-Rim multiplicity,} J. Algebra {\bf 167} (1994), 168–-231.


\bibitem[Mat87]{Matsumura}
Matsumura, H., ``Commutative ring theory.'' Cambridge University Press, 1987. 


\bibitem[McA80]{McAdam}
McAdam, S., {\it Asymptotic prime divisors and analytic spread}, Proc. Amer. Math. Soc., {\bf 90} (1980), 555--559.


\bibitem[McE79]{Eakin}
McAdam S., Eakin P., {\it The asymptotic Ass,} J. Algebra, {\bf 61} (1979), 71–-81.


\bibitem[Ran20]{Rangachev2}
Rangachev, A., {\it Local volumes, equisingularity and generalized smoothability,}\\ \url{https://hal.archives-ouvertes.fr/hal-02429933/document}. 

\bibitem[Ran]{Rangachev3}
Rangachev, A., {\it A valuation theorem for Noetherian rings,} (in preparation). 

\bibitem[Ran18]{Rangachev}
Rangachev, A., {\it Associated points and integral closure of modules}, Journal of Algebra, {\bf 508} (2018), 301--338.

\bibitem[Ratl76]{Ratliff}
Ratliff, L. J. Jr., {\it On prime divisors of $I^{n}$, $n$ large,} Michigan Math. J.,{\bf 23} (1976), 337--352.

\bibitem[Ratl72]{Ratliff2}
Ratliff, L.\ J.\ Jr., {\it On prime divisors of the integral closure of a principal ideals,} J.\ Reine Angew.\ Math.\ {\bf 255} (1972), 210--220. 

\bibitem[R81]{Rees81}
Rees, D., {\it Rings associated with ideals and analytic spread,} Math. Proc. Cambridge Philos. Soc., {\bf 89} (1981), 423–-432.


\bibitem[Stks]{Stacks}
The {Stacks Project Authors}, {\itshape Stacks Project}, \url{http://stacks.math.columbia.edu}, 2016. 

\bibitem[SH06]{Huneke}
Swanson, I., and Huneke, C., ``Integral closure of ideals, rings, and modules.'' London Mathematical Society Lecture Note Series, vol. 336, Cambridge University Press, Cambridge, 2006.

\bibitem[Sh16]{Sharp}
Sharp, R., {\it David Rees, FRS 1918--2013,} Bull. London Math. Soc., {\bf 48} (3) (2016), 557--576.

\bibitem[SUV01]{SUV}
Simis, A., Ulrich, B., Vasconcelos, W., {\it Codimension, multiplicy and integral extensions,} Math. Proc. Camb. Phil. Soc., {\bf 130} (2001), 237--257.
\end{thebibliography}
\end{document}